\documentclass[11pt]{amsart}

\usepackage{amsmath,amssymb,amsthm,mathtools}
\usepackage[hidelinks]{hyperref}

\newcommand{\ZZ}{\mathbb Z}

\newtheorem{theorem}{Theorem}[section]
\newtheorem{proposition}[theorem]{Proposition}
\newtheorem{corollary}[theorem]{Corollary}
\theoremstyle{remark}
\newtheorem{remark}[theorem]{Remark}

\begin{document}

\title[No weakly factor-universal cellular automaton]
{No weakly factor-universal cellular automaton}
\author{Maja Gw\'o\'zd\'z}
\address{University of Zurich}
\address{ETH Z\"urich}
\email{mgwozdz@ethz.ch}
\date{}

\begin{abstract}
Hochman asked whether there exists a cellular automaton \(F\) such that every
cellular automaton is a factor of \(F\) in the dynamical sense. In particular, we do not require the factor map to commute with the spatial shifts. We show that no such
cellular automaton exists. More generally, if \(F\) weakly factors onto the
radius-zero \(q\)-clock automaton \(C_q^{(k)}\), then every periodic point of
\(F\) has period divisible by \(q\). For a cellular automaton \(F:A^{\ZZ^d}\to A^{\ZZ^d}\), define \(\varphi_F:A\to A\) by \(F(\underline a)=\underline{\varphi_F(a)}\), and let
\(g_F\) be the greatest common divisor of the cycle lengths of \(\varphi_F\).
We prove that if \(C_q^{(k)}\) is a weak factor of \(F\), then \(q\mid g_F\) holds.
It follows that the action of \(F\) on constant configurations yields an explicit divisibility
obstruction to clock weak factors.
\end{abstract}

\maketitle

\section{Introduction}

Hochman \cite{Hochman} asked whether there exists a cellular automaton
\[
F:\Sigma^{\ZZ^d}\to \Sigma^{\ZZ^d}
\]
such that every cellular automaton
\[
G:\Delta^{\ZZ^k}\to \Delta^{\ZZ^k}
\]
is a factor of \(F\) in the dynamical sense. More precisely, there should exist a continuous surjection
\[
\pi:\Sigma^{\ZZ^d}\to \Delta^{\ZZ^k}
\]
that satisfies
\[
\pi\circ F=G\circ \pi,
\]
where we do no require that \(\pi\) commute with the spatial shifts. We follow the terminology of \cite{JalonenKari} and refer to such a map as a weak factor map.

For integers \(q\ge2\) and \(k\ge1\), let
\[
C_q^{(k)}:(\ZZ/q\ZZ)^{\ZZ^k}\to(\ZZ/q\ZZ)^{\ZZ^k}
\]
be the radius-zero \(q\)-clock automaton defined by
\[
(C_q^{(k)}(x))_v=x_v+1 \pmod q
\qquad (v\in\ZZ^k).
\]
For every \(n\ge1\),
\[
(C_q^{(k)})^n(x)=x+n \pmod q
\]
holds coordinatewise. Moreover, every point of \(C_q^{(k)}\) has exact period \(q\).

The question has a negative answer, and the obstruction is elementary.
The idea is that factor maps send periodic points to periodic points, while the radius-zero \(q\)-clock automaton has only points of exact period \(q\). This gives
Proposition~\ref{prop:periodic-obstruction} and
Corollary~\ref{cor:no-universal}. The additional point of Theorem~\ref{thm:main} is that this obstruction admits a finite-state reformulation that depends only on the action of \(F\) on constant configurations.

\begin{theorem}\label{thm:main}
Let \(q\ge2\), \(k\ge1\), and let \(F:A^{\ZZ^d}\to A^{\ZZ^d}\) be a cellular
automaton. Define \(\varphi_F:A\to A\) by
\[
F(\underline a)=\underline{\varphi_F(a)}
\qquad (a\in A).
\]
We use \(\underline a\) to denote the constant configuration with value \(a\).
Let \(m_1,\dots,m_r\) be the cycle lengths of \(\varphi_F\), and set
\[
g_F:=\gcd(m_1,\dots,m_r).
\]
If there exists a continuous map
\[
\pi:A^{\ZZ^d}\to(\ZZ/q\ZZ)^{\ZZ^k}
\]
with
\[
\pi\circ F=C_q^{(k)}\circ \pi,
\]
then
\[
q\mid g_F.
\]
\end{theorem}

We mention in passing that our context differs from the rescaled local simulation preorder studied by Boyer and Theyssier \cite{BoyerTheyssier}. For cellular automata \(F_1\) and \(F_2\) of the same dimension, their simulation relation is equivalent to the existence of an onto local map \(\phi\) and integers \(t_1,t_2\ge1\) with
\[
\phi\circ F_1^{t_1}=F_2^{t_2}\circ \phi
\]
\cite[Prop.~1]{BoyerTheyssier}. They showed that for each \(d\ge2\), there is no universal surjective \(d\)-dimensional cellular automaton, and that in dimension \(1\), there exists a universal cellular automaton within the class of cellular automata with a persistent state \cite[Thm.~1, Thm.~2]{BoyerTheyssier}.

\section{Obstructions}

For \(a\in A\), we write \(\underline a\in A^{\ZZ^d}\) for the constant
configuration \(v\mapsto a\).

\begin{proposition}\label{prop:periodic-obstruction}
Let \(q\ge2\), \(k\ge1\), and let \(F:A^{\ZZ^d}\to A^{\ZZ^d}\) be a cellular
automaton. If there exists a continuous map
\[
\pi:A^{\ZZ^d}\to(\ZZ/q\ZZ)^{\ZZ^k}
\]
such that
\[
\pi\circ F=C_q^{(k)}\circ \pi,
\]
then for every \(x\in A^{\ZZ^d}\) and every \(n\ge1\),
\[
F^n(x)=x \quad\Longrightarrow\quad q\mid n.
\]
In particular, every periodic point of \(F\) has least period divisible by \(q\).
\end{proposition}

\begin{proof}
Let \(x\in A^{\ZZ^d}\) and \(n\ge1\) satisfy \(F^n(x)=x\). It follows that
\[
(C_q^{(k)})^n(\pi(x))=\pi(F^n(x))=\pi(x).
\]
For every \(y\in(\ZZ/q\ZZ)^{\ZZ^k}\), we have
\[
(C_q^{(k)})^n(y)=y+n \pmod q
\]
coordinatewise. We conclude that \((C_q^{(k)})^n\) has a fixed point if and only if \(q\mid n\), thus \(q\mid n\) holds.
\end{proof}

\begin{corollary}\label{cor:no-universal}
No weakly factor-universal cellular automaton exists. The same result holds
even if we restrict the class to injective, surjective, or reversible
cellular automata.
\end{corollary}

\begin{proof}
Every cellular automaton has a periodic point, since the finite \(F\)-invariant
set of constant configurations contains a cycle. It follows that no cellular automaton can
weakly factor onto \(C_q^{(k)}\) for every \(q\ge2\). Note that each \(C_q^{(k)}\)
is reversible, so the conclusion follows.
\end{proof}

We now prove Theorem~\ref{thm:main}, which is a finite-state refinement
of Proposition~\ref{prop:periodic-obstruction}. For completeness, note that \(\varphi_F\) is well defined. Observe that \(F\) commutes with every spatial shift, so \(F(\underline a)\) is shift-invariant and thus constant.

\begin{proof}[Proof of Theorem~\ref{thm:main}]
For each cycle of \(\varphi_F\), let us pick a representative \(a_i\in A\). It follows that
\(\underline{a_i}\) is an \(F\)-periodic point of exact period \(m_i\).
We now apply Proposition~\ref{prop:periodic-obstruction} to each \(\underline{a_i}\),
and obtain \(q\mid m_i\) for every \(i\). We conclude that
\[
q\mid \gcd(m_1,\dots,m_r).
\]
\end{proof}

\begin{corollary}\label{cor:prime-obstruction}
For every cellular automaton \(F:A^{\ZZ^d}\to A^{\ZZ^d}\), there exists a prime
\(q\nmid g_F\) such that, for every \(k\ge1\), the reversible cellular automaton
\(C_q^{(k)}\) is not a weak factor of \(F\).
\end{corollary}

\begin{proof}
Let us pick a prime \(q\nmid g_F\). Such a prime exists since \(g_F\) is a positive
integer. Theorem~\ref{thm:main} implies that \(C_q^{(k)}\) is not a weak factor
of \(F\) for any \(k\ge1\).
\end{proof}

\begin{remark}
The divisibility condition of Theorem~\ref{thm:main} is optimal within the family
of clock automata. Indeed, for \(F=C_m^{(d)}\), we have
\[
\varphi_F(a)=a+1 \pmod m,
\]
so \(g_F=m\). If \(q\mid m\), then a coordinatewise reduction modulo \(q\) defines a weak factor map
\[
C_m^{(d)}\longrightarrow C_q^{(d)}.
\]
\end{remark}

\end{document}